\documentclass[reqno]{amsart}

\usepackage[numbers]{natbib}
\usepackage{amssymb}
\usepackage{algorithm}
\usepackage[noend]{algpseudocode}
\usepackage{caption}
\usepackage{footnote}
\usepackage{tikz, tikz-cd}
\usetikzlibrary{matrix, arrows}
\usetikzlibrary{patterns}
\usepackage{wrapfig}
\usepackage{hyperref}
\usepackage{array}
\usepackage{cleveref}
\hypersetup{colorlinks,linkcolor={blue},citecolor={blue},urlcolor={blue}} 

\crefformat{section}{\S#2#1#3}
\crefformat{subsection}{\S#2#1#3}
\crefformat{subsubsection}{\S#2#1#3}

\RequirePackage{tikz}
\RequirePackage{zref-abspage}
\RequirePackage{zref-user}
\RequirePackage{tikz}
\RequirePackage{atbegshi}
\usetikzlibrary{calc}
\RequirePackage{tikzpagenodes}
\RequirePackage{etoolbox}

\newcommand{\Longtop}[2][.424\linewidth]{%
  \leavevmode\hfill\makebox[#1][l]{\begin{minipage}[t]{0.024\linewidth}$\triangleright$\end{minipage}\begin{minipage}[t]{0.4\linewidth}\setlength\parfillskip{0pt}#2\end{minipage}}}
\newcommand{\Longbottom}[2][.424\linewidth]{%
  \leavevmode\hfill\makebox[#1][l]{\hspace{0.024\linewidth}\begin{minipage}[t]{0.4\linewidth}\setlength\parfillskip{0pt}#2\end{minipage}}}
\newcommand{\Top}[2][.424\linewidth]{%
  \leavevmode\hfill\makebox[#1][l]{\begin{minipage}[t]{0.024\linewidth}$\triangleright$\end{minipage}\begin{minipage}[t]{0.4\linewidth}#2\end{minipage}}}
\newcommand{\Bottom}[2][.424\linewidth]{%
  \leavevmode\hfill\makebox[#1][l]{\hspace{0.024\linewidth}#2}}
\newcommand{\Longstate}[2][0cm]{%
  \leavevmode\makebox[#1][l]{\begin{minipage}[t]{18.5em}\setlength\parfillskip{0pt}#2\end{minipage}}}
\newcommand{\Shortstate}[2][0cm]{%
  \leavevmode\makebox[#1][l]{\begin{minipage}[t]{17em}\setlength\parfillskip{0pt}#2\end{minipage}}}
\algnewcommand\algorithmicto{\textbf{to}}
\algnewcommand\RETURN{\State \textbf{return} }
\algrenewcommand\algorithmicdo{\textbf{do}}

\makeatletter
\newcommand*\ALG@lastblockb{b}
\newcommand*\ALG@lastblocke{e}
\apptocmd{\ALG@beginblock}{%
    \ifx\ALG@lastblock\ALG@lastblockb
        \ifnum\theALG@nested>1\relax\expandafter\@firstoftwo\else\expandafter\@secondoftwo\fi{\ALG@tikzborder}{}%
    \fi
    \let\ALG@lastblock\ALG@lastblockb%
}{}{\errmessage{failed to patch}}

\pretocmd{\ALG@endblock}{%
    \ifx\ALG@lastblock\ALG@lastblocke
        \addtocounter{ALG@nested}{1}%
        \addtolength\ALG@tlm{\csname ALG@ind@\theALG@nested\endcsname}%
        \ifnum\theALG@nested>1\relax\expandafter\@firstoftwo\else\expandafter\@secondoftwo\fi{\endALG@tikzborder}{}%
        \addtolength\ALG@tlm{-\csname ALG@ind@\theALG@nested\endcsname}%
        \addtocounter{ALG@nested}{-1}%
    \fi
    \let\ALG@lastblock\ALG@lastblocke%
}{}{\errmessage{failed to patch}}

\tikzset{ALG@tikzborder/.style={line width=0.5pt,black}}

\newcommand*\currenttextarea{current page text area}

\newcommand*{\updatecurrenttextarea}{%
    \if@twocolumn
        \if@firstcolumn
            \renewcommand*{\currenttextarea}{current page column 1 area}%
        \else
            \renewcommand*{\currenttextarea}{current page column 2 area}%
        \fi
    \else
        \renewcommand*\currenttextarea{current page text area}%
    \fi
}

\newcounter{ALG@tikzborder}
\newcounter{ALG@totaltikzborder}
\newenvironment{ALG@tikzborder}[1][]{%
    \ifx&#1&\else
        \tikzset{ALG@tikzborder/.style={#1}}%
    \fi
    \stepcounter{ALG@totaltikzborder}%
    \expandafter\edef\csname ALG@ind@border@\theALG@nested\endcsname{\theALG@totaltikzborder}%
    \setcounter{ALG@tikzborder}{\csname ALG@ind@border@\theALG@nested\endcsname}%
    \tikz[overlay,remember picture] \coordinate (ALG@tikzborder-\theALG@tikzborder);
    \zlabel{ALG@tikzborder-begin-\theALG@tikzborder}%
    \ifnum\zref@extract{ALG@tikzborder-begin-\theALG@tikzborder}{abspage}=\zref@extract{ALG@tikzborder-end-\theALG@tikzborder}{abspage} \else
        \updatecurrenttextarea
        \ALG@drawvline{[shift={(0pt,.5\ht\strutbox)}]ALG@tikzborder-\theALG@tikzborder}{\currenttextarea.south east}{\ALG@thistlm}%
        \newcounter{ALG@tikzborderpages\theALG@tikzborder}%
        \setcounter{ALG@tikzborderpages\theALG@tikzborder}{\numexpr-\zref@extract{ALG@tikzborder-begin-\theALG@tikzborder}{abspage}+\zref@extract{ALG@tikzborder-end-\theALG@tikzborder}{abspage}}%
        \ifnum\value{ALG@tikzborderpages\theALG@tikzborder}>1
            \edef\nextcmd{\noexpand\AtBeginShipoutNext{\noexpand\ALG@tikzborderpage{\theALG@tikzborder}{\the\ALG@thistlm}}}
            \nextcmd
        \fi
    \fi
}{%
    \setcounter{ALG@tikzborder}{\csname ALG@ind@border@\theALG@nested\endcsname}%
    \tikz[overlay,remember picture] \coordinate (ALG@tikzborder-end-\theALG@tikzborder);
    \zlabel{ALG@tikzborder-end-\theALG@tikzborder}%
    \updatecurrenttextarea
    \ifnum\zref@extract{ALG@tikzborder-begin-\theALG@tikzborder}{abspage}=\zref@extract{ALG@tikzborder-end-\theALG@tikzborder}{abspage}\relax
        \ALG@drawvline{[shift={(0pt,.5\ht\strutbox)}]ALG@tikzborder-\theALG@tikzborder}{ALG@tikzborder-end-\theALG@tikzborder}{\ALG@thistlm}%
    \else
        \ALG@drawvline{\currenttextarea.north west}{ALG@tikzborder-end-\theALG@tikzborder}{\ALG@thistlm}%
    \fi
}

\newcommand*{\ALG@drawvline}[3]{
    \begin{tikzpicture}[overlay,remember picture]
        \draw [ALG@tikzborder]
            let \p0 = (\currenttextarea.north west), \p1=(#1), \p2 = (#2)
             in
            (#3+\fboxsep+.5\pgflinewidth+\x0,\y1+\fboxsep+.5\pgflinewidth)
             --
            (#3+\fboxsep+.5\pgflinewidth+\x0,\y2-\fboxsep-.5\pgflinewidth)
        ;
    \end{tikzpicture}%
}

\newcommand{\ALG@tikzborderpage}[2]{
    \updatecurrenttextarea
    \setcounter{ALG@tikzborder}{#1}%
    \ALG@drawvline{\currenttextarea.north west}{\currenttextarea.south east}{#2}%
    \addtocounter{ALG@tikzborderpages\theALG@tikzborder}{-1}%
    \ifnum\value{ALG@tikzborderpages\theALG@tikzborder}>1
        \AtBeginShipoutNext{\ALG@tikzborderpage{#1}{#2}}%
    \fi
    \vspace{-0.5\baselineskip}
}

\def\ALG@tikzbordertext{\the\ALG@tlm}
\makeatother

\RequirePackage{etoolbox}
\makeatletter
\newlength{\ALG@continueindent}
\newcommand*{\ALG@customparshape}{\parshape 2 \leftmargin \linewidth \dimexpr\ALG@tlm+\ALG@continueindent\relax \dimexpr\linewidth+\leftmargin-\ALG@tlm-\ALG@continueindent\relax}
\newcommand*{\ALG@customparshapex}{\parshape 1 \dimexpr\ALG@tlm+\ALG@continueindent\relax \dimexpr\linewidth+\leftmargin-\ALG@tlm-\ALG@continueindent\relax}
\apptocmd{\ALG@beginblock}{\ALG@customparshape\everypar{\ALG@customparshapex}}{}{\errmessage{failed to patch}}
\makeatother

\newtheorem{theorem}{Theorem}[section]
\newtheorem{lemma}[theorem]{Lemma}
\newtheorem{proposition}[theorem]{Proposition}
\newtheorem{corollary}[theorem]{Corollary}

\theoremstyle{definition}
\newtheorem{definition}[theorem]{Definition}
\newtheorem{notation}[theorem]{Notation}

\numberwithin{equation}{section}

\newcommand{\bR}{\mathbb{R}}
\newcommand{\bQ}{\mathbb{Q}}
\newcommand{\bZ}{\mathbb{Z}}
\newcommand{\bN}{\mathbb{N}}
\newcommand{\bq}{\mathbf{q}}
\newcommand{\bx}{\mathbf{x}}
\newcommand{\by}{\mathbf{y}}
\newcommand{\bu}{\mathbf{u}}
\newcommand{\be}{\mathbf{e}}
\newcommand{\bb}{\mathbf{b}}
\newcommand{\bc}{\mathbf{c}}
\newcommand{\M}{\text{M}}
\newcommand{\mM}{\skew{2.5}\check{\rule{0.0cm}{0.22cm}\smash{M}}}
\newcommand{\lcd}{\text{lcd}}
\newcommand{\adj}{\text{adj}}

\begin{document}

\title[Short vector problems and simultaneous approximation]{Reductions between short vector problems and simultaneous approximation}

\thanks{This research was supported by NSF-CAREER CNS-1652238 under the supervision of PI Dr. Katherine E. Stange.}

\author{Daniel E. Martin}
\address{University of California, Davis, CA, United States}
\email{dmartin@math.ucdavis.edu}



\date{\today}

\begin{abstract}
In 1982, Lagarias showed that solving the approximate Shortest Vector Problem also solves the problem of finding good simultaneous Diophantine approximations \cite{lagarias}. Here we provide a deterministic, dimension-preserving reduction in the reverse direction. It has polynomial time and space complexity, and it is gap-preserving under the appropriate norms. We also give an alternative to the Lagarias algorithm by first reducing the version of simultaneous approximation in \cite{lagarias} to one with no explicit range in which a solution is sought.\end{abstract}

\maketitle

\section{Introduction}

Our primary result is to show that a short vector problem reduces deterministically and with polynomial complexity to a single simultaneous approximation problem as presented in the definitions below. We use $\min^\times$ to denote the nonzero minimum, $\{\bx\}\in(-1/2,1/2\hspace{0.04cm}]^n$ to denote the fractional part of $\bx\in\bR^n$, and $[x]$ to denote the set $\{1,...,\lfloor x\rfloor\}$ for $x\in\bR$.

\begin{definition}\label{def:svp}A \emph{short vector problem} takes input $\alpha\in[1,\infty)$ and nonsingular $M\in\M_n(\bZ)$. A valid output is $\bq_0\in\bZ^n$ with $0<\|M\bq_0\| \leq \alpha\min^{\times}_{\bq\in\bZ^n}\!\|M\bq\|$. Let \textsc{svp} denote an oracle for such a problem.\end{definition}

\begin{definition}\label{def:gda}A \emph{good Diophantine approximation problem} takes input $\alpha,N\in[1,\infty)$ and $\bx\in\bQ^n$. A valid output is $q_0\in[\alpha N]$ with $\|\{q_0\bx\}\| 
\leq\alpha\min_{q\in[N]}\!\|\{q\bx\}\|$. Let \textsc{gda} denote an oracle for such a problem.\end{definition}

Our reduction asserts that if we can find short vectors in a very restricted family of lattices then we can find them in general, since behind a good Diophantine approximation problem is the lattice generated by $\bZ^n$ and one additional vector, $\bx$.

Literature more commonly refers to a short vector problem as a \textit{shortest vector problem} when $\alpha = 1$ and an \textit{approximate shortest vector problem} otherwise (often unrestricted to sublattices of $\bZ^n$, though we have lost no generality). A brief exposition can be found in \cite{nguyen}. See \cite{galbraith} or \cite{goldwasser} for a more comprehensive overview, \cite{peikert} for a focus on cryptographic applications, \cite{kumar} for a summary of hardness results, and \cite{bernstein} for relevance and potential applications to post-quantum cryptography. 

Regarding simultaneous approximation, Brentjes highlights several algorithms in \cite{brentjes}. For a sample of applications to attacking clique and knapsack-type problems see \cite{frank}, \cite{lagarias2}, and \cite{shamir}. Examples of cryptosystems built on the hardness of simultaneous approximation are \cite{armknecht}, \cite{baocang}, and \cite{inoue}. This version is taken from \cite{chen} and \cite{rossner}.

The reduction, given in Algorithm \ref{alg:agrawal}, preserves the gap $\alpha$ when the $\ell_{\infty}$-norm is used for both problems. This means the short vector problem defined by $\alpha$ and $M$ is solved by calling $\textsc{gda}(\alpha,\bx, N)$ for some $\bx\in\bQ^n$ and $N\in\bR$. It reverses a 1982 result of Lagarias, which reduces a good Diophantine approximation problem to \textsc{svp}. (See Theorem B in \cite{lagarias}, which refers to the problem as \textit{good simultaneous approximation}. We borrow its name from \cite{chen} and \cite{rossner}.) Though there is an important contextual distinction: \cite{lagarias} relates simultaneous approximation under the $\ell_{\infty}$-norm to lattice reduction under the $\ell_2$-norm, whereas \emph{all reductions in this paper assume a consistent norm}. 

Under Lagarias' (and the most common) setup ---the $\ell_{\infty}$-norm for \textsc{gda} and the $\ell_2$-norm for \textsc{svp}---we are not the first to go the other direction. In a seminar posted online from July 1, 2019, Agrawal presented an algorithm achieving this reduction which was complete less some minor details \cite{agrawal}. Tersely stated, he takes an upper triangular basis for a sublattice of $\bZ^n$ and transforms it inductively, using integer combinations and rigid rotations with two basis vectors at time, into a lattice (a rotated copy of the original) whose short vectors can be found via simultaneous approximation. The short vector problem defined by $\alpha$ and $M$ gets reduced to $\textsc{gda}(\alpha/\sqrt{2n},\bx,N)$, called multiple times in order to account for the unknown minimal vector length which is used to determine $\bx$. 

In contrast, the reduction here takes a completely different approach. It finds a sublattice which is nearly scaled orthonormal, so that only one additional vector is needed to generate the original lattice. This extra vector is the input for $\textsc{gda}$. We note that when switching between norms, our reduction is also not gap-preserving. To use Algorithm \ref{alg:agrawal} to solve a short vector problem with respect to the $\ell_2$-norm via \textsc{gda} with respect to the $\ell_{\infty}$-norm, the latter must be executed with the parameter $\alpha/\!\sqrt{n}$ to account for the maximum ratio of nonzero norms $\|\bq\|_2/\|\bq\|_{\infty}$.

The relationship between the two problems in Definitions \ref{def:svp} and \ref{def:gda} will be studied through the following intermediary.

\begin{definition}\label{def:sap}A \emph{simultaneous approximation problem} takes input $\alpha\in[1,\infty)$ and $\bx\in\bQ^n$. A valid output is $q_0\in\bZ$ with $0 < \|\{q_0\bx\}\| 
\leq\alpha\min^{\times}_{q\in\bZ}\!\|\{q\bx\}\|$. Let \textsc{sap} denote an oracle for such a problem.\end{definition}

This problem prohibits only the trivial solution, the least common denominator of $\bx$'s entries, while ``$N$" in a good Diophantine approximation problem is generally more restrictive.

Section \ref{sec:2} explores the relationship between the two versions of simultaneous approximation given in Definitions \ref{def:gda} and \ref{def:sap}. Among the results, only Proposition \ref{prop:equivalent} in Subsection \ref{ss:lose} is required to verify the final reduction of a short vector problem to either version of simultaneous approximation. Subsection \ref{ss:recover} contains Algorithm \ref{alg:diophantus}. It reduces a good Diophantine approximation problem to polynomially many \textsc{sap} calls, each executed with the parameter $\alpha/3.06$. So while this reduction is not gap-preserving, the inflation is independent of the input.

Section \ref{sec:3} reduces both versions of simultaneous approximation to \textsc{svp}. It begins with Algorithm \ref{alg:lagarias}, which solves Definition \ref{def:sap}'s version. We remark at the end of Subsection \ref{ss:start} how this reduction adapts to the inhomogeneous forms of these problems, meaning the search for $q_0\in\bZ$ or $\bq_0\in\bZ^n$ that makes $q_0\bx - \by$ or $M\bq_0 - \by$ small for some $\by\in\bQ^n$. (In this case the latter is known as the \textit{approximate closest vector problem}, exposited in chapter 18 of \cite{galbraith}, for example.) Then Subsection \ref{ss:lagarias} combines Algorithms \ref{alg:diophantus} and \ref{alg:lagarias} to solve Definition \ref{def:gda}'s version of simultaneous approximation using \textsc{svp}. This is our alternative to the Lagarias reduction.

Finally, Algorithm \ref{alg:agrawal} in Section \ref{sec:4} reduces a short vector problem to \textsc{gda} or \textsc{sap}. It also adapts to the inhomogeneous versions of \textsc{svp} and \textsc{sap} (not \textsc{gda}, as mentioned at the end of Subsection \ref{ss:discuss}). In Corollary \ref{cor:nphard} we observe that Algorithm \ref{alg:agrawal} facilitates a simpler proof that \textsc{gda} is NP-hard under an appropriate bound on $\alpha$, a result first obtained in \cite{chen}. Then we combine Algorithms \ref{alg:lagarias} and \ref{alg:agrawal} in Subsection \ref{ss:reduction} to solve a simultaneous approximation problem with \textsc{gda}. In particular, we give all six reductions among the defined problems, as shown in the diagram below.

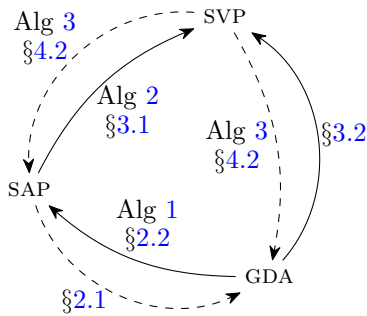
\begin{wrapfigure}{l}{0.41\textwidth}
\setlength{\abovecaptionskip}{0pt}
\setlength{\belowcaptionskip}{-8pt}
    \vspace{-0.2cm}
    \centering
    \begin{tikzpicture}[scale = 2.9]
    \draw(0,0,0) node (0) {\textsc{sap}};
    \draw(1.1,-0.4,0) node (1) {\textsc{gda}};
    \draw(0.896,0.793,0) node (2) {\textsc{svp}};
    \draw(0.07,0.76,0) node (3) {Alg \ref{alg:agrawal}};
    \draw(0.07,0.63,0) node (4) {\cref{ss:reduction}};
    \draw(0.95,0.26,0) node (5) {Alg \ref{alg:agrawal}};
    \draw(0.95,0.13,0) node (6) {\cref{ss:reduction}};
    \draw(1.44,0.25,0) node (7) {\cref{ss:lagarias}};
    \draw(0.45,0.42,0) node (8) {Alg \ref{alg:lagarias}};
    \draw(0.45,0.29,0) node (9) {\cref{ss:start}};
    \draw(0.54,-0.10,0) node (10) {Alg \ref{alg:diophantus}};
    \draw(0.54,-0.23,0) node (11) {\cref{ss:recover}};
    \draw(0.25,-0.515,0) node (12) {\cref{ss:lose}};

    \draw [bend right=49,-{Stealth[round,scale=1.3]}, dashed] (0) to (1);
    \draw [bend right=49,-{Stealth[round,scale=1.3]}] (1) to (2);
    \draw [bend right=49,-{Stealth[round,scale=1.3]}, dashed] (2) to (0);
    \draw [bend left=20,-{Stealth[round,scale=1.3]}] (0) to (2);
    \draw [bend left=20,-{Stealth[round,scale=1.3]}, dashed] (2) to (1);
    \draw [bend left=20,-{Stealth[round,scale=1.3]}] (1) to (0);
    \end{tikzpicture}
    \caption{Algorithm and subsection numbers for reductions.}\label{fig:1}
\end{wrapfigure}

The two reductions in Figure \ref{fig:1} without algorithm numbers are achieved by following the two arrows that combine to give the same source and target. \textit{Dashed arrows indicate a norm restriction. Each must be executed under either the $\ell_1$, $\ell_2$, or $\ell_{\infty}$-norm.} However, we point out in Subsection \ref{ss:discuss} how the restriction can be alleviated to any $\ell_p$-norm provided we accept additional gap inflation by a constant arbitrarily close to $1$.

The results are summarized by the following table. It uses $m$ and $d$ to denote the maximal magnitude among input integers and the least common denominator of the input vector, respectively. The matrix or vector dimension is $n$, and $p$ defines the norm. Trivial cases that cause logarithms to equal $0$ are ignored. Column descriptions follow.

\vspace{-0.0cm}

\begin{table}[h]
\setlength{\belowcaptionskip}{5pt}
\setlength{\abovecaptionskip}{-12pt}
\begin{tabular}{|>{\centering\arraybackslash}m{1.85cm}|>{\centering\arraybackslash}m{2cm}>{\centering\arraybackslash}m{1.5cm}>{\centering\arraybackslash}m{1.5cm}>{\centering\arraybackslash}m{1.8cm}|}
\hline
\vspace{0.03cm}\textbf{Reduction} & \vspace{0.03cm}\textbf{Operations} & \vspace{0.03cm}\textbf{Integers} & \vspace{0.03cm}\textbf{Inflation} & \vspace{0.03cm}\textbf{Calls}\\
\hline
$\textsc{gda}\to\textsc{sap}$ & $n\log m$ & $n\log m$ & $3.06$ & $\lceil\log_2\! d/\alpha N\rceil$\\
$\textsc{sap}\to\textsc{svp}$ & $(n+\log m)^2$ & $n\log m$ & $1$ & $1$\\
$\textsc{gda}\to\textsc{svp}$ & $(n+\log m)^2$ & $n\log m$ & $3.06$ & $\lceil\log_2\! d/\alpha N\rceil$\\
$\textsc{svp}\to\textsc{gda}$ & $n^4\log mn$ & $n^4\log mn$ & $n^{1/p}$ & 1\\
$\textsc{svp}\to\textsc{sap}$ & $n^4\log mn$ & $n^4\log mn$ & 1 & 1\\
$\textsc{sap}\to\textsc{gda}$ & $n^5\log m$ & $n^5\log m$ & $n^{1/p}$ & 1\\
\hline
\end{tabular}
\caption{Summary of reduction complexities and gap inflations.}
\label{tab:results}
\end{table}

\vspace{0.15cm}

\noindent\textbf{Operations:} Big-$O$ bound on the number of arithmetic operations per oracle call. 

\noindent\textbf{Integers:} Big-$O$ bound on the length of integers used throughout the reduction.

\noindent\textbf{Inflation:} Maximum gap inflation. For example, to solve a good Diophantine approximation problem with some $\alpha$ using Algorithm \ref{alg:diophantus}, \textsc{sap} is called with $\alpha/3.06$.

\noindent\textbf{Calls:} Upper bound on the number of required calls to the oracle.

\section{\label{sec:2}Versions of simultaneous approximation}

\subsection{SAP to GDA}\label{ss:lose} Rather than give a complete reduction from a simultaneous approximation problem to \textsc{gda}, which is postponed until the end of Subsection \ref{ss:reduction}, the purpose of this subsection is to observe a condition on the input that makes these two versions of simultaneous approximation nearly equivalent.

\begin{proposition}\label{prop:equivalent}Suppose the $i^\text{th}$ coordinate of $\bx$ is of the form $x_i=1/d$, where $d\in\bN$ makes $d\bx\in\bZ^n$. Under an $\ell_{p}$-norm, $\textsc{gda}(\alpha,\bx,N)$ solves the simultaneous approximation problem defined by $\alpha n^{1/p}$ and $\bx$ with $N = d/2\alpha$.\end{proposition}

\begin{proof}Let $q_{\min}\in [d/2]$ be such that $\|\{q_{\min}\bx\}\|$ is the nonzero minimum. A vector's fractional part is in $(-1/2,1/2\hspace{0.04cm}]^n$, making its length at most $n^{1/p}/2$. So we may assume that $\|\{q_{\min}\bx\}\| < 1/2\alpha$, since otherwise every integer in $[N] = [d/2\alpha]$ solves the simultaneous approximation problem defined by $\alpha n^{1/p}$ and $\bx$.

Under an $\ell_p$-norm, $\|\{q_{\min}\bx\}\|$ is an upper bound for its $i^\text{th}$ coordinate, $q_{\min}/d$. Combined with the assumption $\|\{q_{\min}\bx\}\| < 1/2\alpha$, this gives $q_{\min}\in [d/2\alpha]=[N]$, which implies $\min_{q\in[N]}\!\|\{q\bx\}\| \leq \min^{\times}_{q\in\bZ}\!\|\{q\bx\}\|$. And because $\alpha N < d$, it is guaranteed that $\textsc{gda}(\alpha,\bx,N)$ is not a multiple of $d$.\end{proof}

Note that without an assumption on $\bx$ like the one used in this proposition, there is no natural choice for $N$ that makes $\textsc{gda}$ solve a simultaneous approximation problem. If we set it too small, say with $N < d/2$, then $\min_{q\in[N]}\!\|\{q\bx\}\|$ may be unacceptably larger than $\min^\times_{q\in\bZ}\!\|\{q\bx\}\|$, potentially making \textsc{gda}'s approximation poor. If we set it too large, say with $N \geq d/\alpha$, then \textsc{gda} may return $d$, which is not a valid output for the initial simultaneous approximation problem.

To get around this, our strategy is to first reduce a simultaneous approximation problem to $\textsc{svp}$ with Algorithm \ref{alg:lagarias}. Then in Algorithm \ref{alg:agrawal}, which reduces a short vector problem to \textsc{sap}, we are careful to produce an input vector for the oracle that satisfies the hypothesis of Proposition \ref{prop:equivalent} in order to admit \textsc{gda}.

\subsection{GDA to SAP}\label{ss:recover} Let $d$ continue to denote the least common denominator of $\bx$. The problem faced in this reduction is that outputs for a good Diophantine approximation problem are bounded by $\alpha N$, which may be smaller than $d/2$. This leaves no guarantee that $\textsc{sap}(\alpha,\bx)$, call this integer $d_1\in[d/2]$, is a solution. But knowing that $\bx$ is very near a rational vector $\bx_1$ with least common denominator $d_1$ allows us to call \textsc{sap} again, now on $\bx_1$ to get $d_2\in[d_1/2]$. This is the least common denominator of some $\bx_2$ near $\bx_1$, and we continue in this fashion until the output is at most $\alpha N$. To get $d_i\in[d_{i-1}/2]$, we adopt the convention that modular reduction returns an integer with magnitude at most half the modulus.

\vspace{0.105cm}

\begin{algorithm}\caption{A reduction from a good Diophantine approximation problem to multiple calls to \textsc{sap} under a consistent norm.}\label{alg:diophantus}
\begin{flushleft}
\hspace*{\algorithmicindent}\textbf{input:} $\alpha,N\in[1,\infty)$, $\bx=(x_1,...,x_n)\in\bQ^n$\\
\hspace*{\algorithmicindent}\textbf{output:} $q_0\in[\alpha N]$ with $\|\{q_0\bx\}\| \leq \alpha\min_{q\in[N]}\!\|\{q\bx\}\|$
\end{flushleft}
\begin{algorithmic}[1]
    \State $d\gets \lcd(x_1,...,x_n)>0$
    \While{$d>\alpha N$}
        \State $d\gets |\textsc{sap}(\alpha/3.06,\bx)\,\text{mod}\,d|$\Top{good, but large denominator}
        \State $\bx\gets \bx-\{d\bx\}/d$\Longtop{now $\lcd(\bx)=d$, at most half of}
        \Statex \vspace{-\baselineskip}
    \EndWhile
    \State \Return $d$\Bottom{the previous iteration's $\lcd$}
\end{algorithmic}
\end{algorithm}

\vspace{-0.105cm}

\begin{proposition}\label{prop:diop_works}The output of Algorithm \ref{alg:diophantus} solves the initial good Diophantine approximation problem.\end{proposition}

\begin{proof}Let $d_i$ and $\bx_i$ denote the values of $d$ and $\bx$ after $i$ \textbf{while} loop iterations have been completed. In particular, $d_0$ and $\bx_0$ are defined by the input. Also let $I+1$ be the total number of iterations executed, so the output is $d_{I+1}$. 

The triangle inequality gives \begin{equation}\label{eq:1}\|\{d_{I+1}\bx\}\|\leq \|\{d_{I+1}\bx_I\}\| + d_{I+1}\sum_{i=1}^{I}\|\bx_i - \bx_{i-1}\|.\end{equation} With $\lambda_i=\min_{q\in[N]}\!\|\{q\bx_i\}\|$, the choice of $d_{I+1}$ bounds the first summand by $\alpha\lambda_{I}/c$, where $c=3.06$ in the algorithm but is left undetermined for now. Similarly, the choice of $d_i=\textsc{sap}(\alpha/c,\bx_{i-1})$ and the fact that $d_{i-1} > \alpha N\geq N$ make \begin{equation}\label{eq:2}\|\bx_i - \bx_{i-1}\|=\frac{\|\{d_i\bx_{i-1}\}\|}{d_i}\leq \frac{\alpha\min^\times_{q\in\bZ}\!\|\{q\bx_{i-1}\}\|}{cd_i}\leq \frac{\alpha\lambda_{i-1}}{cd_i}.\end{equation} So to bound (\ref{eq:1}) it must be checked that the $\lambda_i$'s are not too large. To this end, fix some $i\leq I$ and let $q_{\min}\in[N]$ satisfy $\|\{q_{\min}\bx_{i-1}\}\| = \lambda_{i-1}$. Then we have the following upper bound on $\lambda_i$, where the three inequalities are due to the triangle inequality, inequality (\ref{eq:2}), and $q_{\min}\leq N < d_I/\alpha \leq d_i/2^{I-i}\alpha$, respectively: $$ \|\{q_{\min}\bx_i\}\|\leq \lambda_{i-1} + q_{\min}\|\bx_i - \bx_{i-1}\|\leq \lambda_{i-1}\left(1 + \frac{\alpha q_{\min}}{cd_i}\right)< \lambda_{i-1}\left(1 + \frac{1}{2^{I-i}c}\right).$$ Inductively, this gives \begin{equation}\label{eq:3}\lambda_i<\lambda_0\prod_{j=1}^i\left(1 + \frac{1}{2^{I-j}c}\right).\end{equation} Now the three numbered inequalities above can be combined to get $$\|\{d_{I+1}\bx\}\|\leq \frac{\alpha d_{I+1}}{c}\sum_{i=0}^{I}\frac{\lambda_i}{d_{i+1}}\leq \frac{\alpha}{c}\sum_{i=0}^I\frac{\lambda_i}{2^{I-i}}\leq \frac{\alpha\lambda_0}{c}\sum_{i=0}^I\frac{1}{2^{I-i}}\prod_{j=1}^i\left(1 + \frac{1}{2^{I-j}c}\right).$$ Thus the output approximation quality, $\|\{d_{I+1}\bx\}\|$, is at most $\alpha\min_{q\in[N]}\!\|\{q\bx\}\|$ $=\alpha\lambda_0$ provided $c$ satisfies $$1\geq \frac{1}{c}\sum_{i=0}^\infty\frac{1}{2^i}\prod_{j=i}^\infty\left(1 + \frac{1}{2^jc}\right).$$ This justifies our choice of $c=3.06$ in line 3.\end{proof}

\begin{proposition}\label{prop:diop_speed}Let $m>1$ be the maximum magnitude among integers defining $\bx$, and let $d>1$ be its least common denominator. The reduction in Algorithm \ref{alg:diophantus} requires an initial $O(n\log m)$ operations plus $O(n)$ operations for each call to \textsc{sap}, of which there are at most $\lceil \log_2(d/\alpha N)\rceil$, on integers of length $O(n\log m)$.\end{proposition}

\begin{proof}Repeatedly applying the Euclidean algorithm computes $d$ with $O(n\log m)$ operations on integers of length $O(n\log m)$. Modular reduction in line 3 decreases each successive least common denominator by at least a factor of $1/2$. This bounds the number of \textbf{while} loop iterations by $\lceil \log_2(d/\alpha N)\rceil$.\end{proof}

\section{\label{sec:3}Reducing to svp}

First we restrict attention to Definition \ref{def:sap}'s version of simultaneous approximation (\textsc{sap}) in Algorithm \ref{alg:lagarias}. Then we will compare the combination with Algorithm \ref{alg:diophantus} to Lagarias' reduction in \cite{lagarias} from Definition \ref{def:gda}'s version (\textsc{gda}).

\subsection{SAP to SVP}\label{ss:start} Here we replace the $n+1$ vectors associated to simultaneous approximation, namely $\bx$ and a basis for $\bZ^n$, with $n$ vectors generating the same lattice. There are algorithms for which this is a byproduct, like Pohst's modified (to account for linearly dependent vector inputs) LLL algorithm \cite{pohst} or Kannan and Bachem's Hermite normal form algorithm \cite{kannan}. But as a consequence of achieving additional basis properties, they are more complicated and require more operations than necessary. We briefly present an alternative because the improved time complexity is relevant to the next subsection.

\vspace{0.115cm}

\begin{algorithm}\caption{A gap-preserving reduction from a simultaneous approximation problem to one call to \textsc{svp} under a consistent norm.}\label{alg:lagarias}
\begin{flushleft}
\hspace*{\algorithmicindent}\textbf{input:} $\alpha\in[1,\infty)$, $\bx=(x_1,...,x_n)\in\bQ^n$\\
\hspace*{\algorithmicindent}\textbf{output:} $q_0\in\bZ$ with $0 < \|\{q_0\bx\}\| \leq \alpha\min^{\times}_{q\in\bZ}\!\|\{q\bx\}\|$
\end{flushleft}
\begin{algorithmic}[1]
    \State $d\gets \lcd(x_1,...,x_n)$
    \State\Longstate{$x_n\gets x_n+a$ with $a$ an integer that makes}\Longtop{make sure $d\bx$ extends to a basis}
    \Statex $\text{gcd}(dx_1,...,dx_{n-1},d(x_n+a))=1$\Bottom{for $\bZ^n$}
    \State\Longstate{$M\hspace{-0.08cm}\gets\hspace{-0.03cm}$ matrix in $\text{SL}_n(\bZ)$ with first column $d\bx$}
    \State\Longstate{$M\gets M$ with last $n-1$ columns scaled by $d$}\Longtop{generates scaled original lattice}
    \State \Return $\textsc{svp}(\alpha,M)_1$\Top{first coordinate is a solution}
\end{algorithmic}
\end{algorithm}

\vspace{-0.115cm}

\begin{proposition}\label{prop:lag_works}The output of Algorithm \ref{alg:lagarias} solves the initial simultaneous approximation problem.\end{proposition}

\begin{proof}First note that $a$ in line 2 exists. As $d$ is the \textit{least} common denominator, $\text{gcd}(dx_1,...,dx_n)$ and $d$ are coprime. So take $a$ to be divisible by those primes which divide $\text{gcd}(dx_1,...,dx_{n-1})$ but not $dx_n$. Also, since $a$ is an integer, the new value of $\bx$ defines the same simultaneous approximation problem as the input.

Coprime entries means $d\bx$ extends to some $M\in\text{SL}_n(\bZ)$. (One method is mentioned in the next proof.) The columns of $dM$ generate $d\bZ^n$, so the same is true if we only scale the last $n-1$ columns by $d$. In particular, the columns of the new $M$ in line 4 generate $d\bx$ and $d\bZ^n$, which in turn generate each column. Thus $M$ defines a basis for the original simultaneous approximation lattice scaled by $d$.

Finally, the last $n-1$ columns of $M$ are vectors in $d\bZ^n$, so that $M\textsc{svp}(\alpha,M)\equiv \textsc{svp}(\alpha,M)_1d\bx\text{ mod }d\bZ^n$. This verifies that $\textsc{svp}(\alpha,M)_1$ is the integer we seek.\end{proof}

\begin{proposition}\label{prop:lag_speed}Let $m>1$ be the maximum magnitude among integers defining $\bx$. The reduction in Algorithm \ref{alg:lagarias} requires $O((n+\log m)^2)$ operations on integers of length $O(n\log m)$.\end{proposition}

\begin{proof}As with Algorithm \ref{alg:diophantus}, line 1 requires $O(n\log m)$ operations on and resulting in integers of length $O(n\log m)$.

Skipping line 2 for now, the $i^{\text{th}}$ column (for $i\geq 2$) of $M$ in line 3 can be set to $$\left(\frac{b_1dx_1}{\text{gcd}(dx_1,...,dx_{i-1})},...,\frac{b_1dx_{i-1}}{\text{gcd}(dx_1,...,dx_{i-1})},b_2,0,...,0\right)$$ (transposed), where $b_2\text{gcd}(dx_1,...,dx_{i-1})-b_1dx_i = \text{gcd}(dx_1,...,dx_i)$. The determinant of the top, left $i\times i$ minor is then $\text{gcd}(dx_1,...,dx_i)$ by induction. To find $b_1$ and $b_2$ we execute the Euclidean algorithm on $\text{gcd}(d_ix_1,...,d_ix_{i-1})$ and $d_ix_i$, where $d_i=\text{lcd}(x_1,...,x_i)$. But $\text{gcd}(d_ix_1,...,d_ix_{i\hspace{-0.02cm}-\hspace{-0.02cm}1})$ is at most $m$ times $\text{gcd}(d_{i\hspace{-0.02cm}-\hspace{-0.02cm}1}x_1,...,d_{i\hspace{-0.02cm}-\hspace{-0.02cm}1}x_{i\hspace{-0.02cm}-\hspace{-0.02cm}1})$, which divides the greatest common divisor of the numerators of $x_1,...,x_{i-1}$. So for each $i$ the Euclidean algorithm needs $O(\log m)$ operations.

Before computing the last column of $M$, we find $a$ in line 2 to ensure a determinant of $1$. As discussed in the last proof, we can start with $a=\text{gcd}(dx_1,...,dx_{n-1})$ and replace it with $a / \text{gcd}(a,dx_n)$ until nothing changes. This requires $O(\log a) = O(\log m)$ executions of the Euclidean algorithm, each taking $O(\log m)$ operations.

Scaling all but the first column by $d$ in line 4 takes $O(n^2)$ operations.\end{proof}

We remark that this algorithm adapts to inhomogeneous forms of these problems. To find $q_0\in\bZ$ with $\|\{q_0\bx - \by\}\| \leq \alpha\min_{q\in\bZ}^{\times}\!\|\{q\bx - \by\}\|$ when $q\bx - \by\in \bZ^n$ has no solution, we can perform the same reduction and finish by calling an oracle which solves the approximate \textit{closest} vector problem defined by $\alpha$, $M$, and $d\by$.

\subsection{GDA to SVP}\label{ss:lagarias} {Combining Algorithms \ref{alg:diophantus} and \ref{alg:lagarias} gives an alternative to the Lagarias reduction from good Diophantine approximation to \textsc{svp} in \cite{lagarias}. We execute Algorithm \ref{alg:diophantus}, but use Algorithm \ref{alg:lagarias} to compute $\textsc{sap}(\alpha/3.06,\bx)$ in line 3. By Proposition \ref{prop:diop_speed}, this requires at most $\lceil \log_2(d/\alpha N)\rceil$ calls to $\textsc{svp}$. And Proposition \ref{prop:lag_speed} states that each call requires $O((n+\log m)^2)$ operations on integers of length $O(n\log m)$.\unskip\parfillskip 0pt \par}

Recall that switching from $\ell_{2}$ to $\ell_{\infty}$ decreases a nonzero norm by at most a factor of $1/\sqrt{n}$. In particular, by executing this combination of Algorithms \ref{alg:diophantus} and \ref{alg:lagarias} with respect to the $\ell_2$-norm, we get an $\ell_{\infty}$ solution to the initial good Diophantine approximation problem provided we use $\alpha/3.06\sqrt{n}$ for $\textsc{svp}$.

Lagarias achieves this reduction with the now well-known trick from \cite{lenstra} of reducing the lattice generated by $\bZ^n$ and $\bx$, bumped up a dimension by putting $0$ in every $(n+1)^{\text{th}}$ coordinate but $\bx$'s. The ideal value for the last coordinate of $\bx$, which is guessed at using $\lfloor n+\log_2dN\rfloor$ calls of the form $\textsc{svp}(\alpha/\sqrt{5n}, M)$ for varying $M$, is $\min_{q\in[N]}\!\|\{q\bx\}\|/N$. (The gap inflation approaches $\sqrt{n}$ as our guesses get better.) The Lagarias reduction requires an initial $O(n\log m)$ arithmetic operations to compute the least common denominator, then only one additional operation per call. The integers involved have input length $O(\log m^nN)$. 

Whether the benefit of fewer calls to \textsc{svp} outweighs the increased operations per call depends on the complexity of the oracle. Ours is an asymptotic improvement when the number of operations performed by \textsc{svp} exceeds $O((n+\log m)^2)$. 

\vspace{0.08cm}

\section{\label{sec:4}Reducing to gda or sap}

\vspace{-0.06cm}

We focus first on the reduction to \textsc{sap}.

\subsection{Intuition}\label{ss:intuition} Consider an input matrix $M\in\M_n(\bZ)$ for a short vector problem. Let $d = \det M$, and let $\be_1,...,\be_n$ denote the standard basis vectors for $\bZ^n$. If there were one vector, call it $\bb\in \bZ^n$, for which the set $\{M\bb, d\be_1,...,d\be_n\}$ generated the columns of $M$, our reduction would just amount to finding it. This is exactly the setup for simultaneous approximation: $n+1$ vectors, $n$ of which are scaled orthonormal. A solution could be obtained by doing simultaneous approximation on $M\bb/d$, scaling the resulting short vector by $d$, and applying $M^{-1}$ (to comply with Definition \ref{def:svp}). Unfortunately, unless $n\leq 2$ or $d = \pm1$, such a $\bb$ does not exist. Indeed, the adjugate matrix, $\adj\, M = d\,M^{-1}$, has at most rank 1 over $\bZ/p\bZ$ for a prime $p$ dividing $d$. So at least $n-1$ additional vectors are required to have full rank modulo $p$, a prerequisite to having full rank over $\bQ$. But asking that $M\bb$ generate the columns of $M$ alongside $d\be_1,...,d\be_n$ is equivalent to asking that $\bb$ generate $\bZ^n$ alongside the columns of $\adj\,M$. 

What mattered is the matrix with columns $d\be_1,...,d\be_n$ being scaled orthonormal. As such, multiplying by it or its inverse has no effect on a vector's relative length. So we plan to find a different set of $n$ column vectors---a set for which just one additional $M\bb$ is needed to generate the original lattice---which is nearly scaled orthonormal, making the effect of its corresponding matrix multiplication on $\alpha$ negligible. The initial short vector problem becomes a search for an integer combination of $M\bb$ and these columns, say $\bc_1,...,\bc_n$. We can then solve the simultaneous approximation problem defined by $\alpha$ and $[\bc_1\,\cdots\,\bc_n]^{-1}M\bb$. This works as long as multiplying by $[\bc_1\,\cdots\,\bc_n]$ changes the ratio between the lengths of the shortest vector and our output by less than whatever is afforded by the fact that lattice norms form a discrete set. 

An arbitrary lattice may have all of its scaled orthonormal sublattices contained in $d\bZ^n$. So as candidates for the matrix $[\bc_1\,\cdots\,\bc_n]$, we look for something of the form $cd\,\text{Id} + MA = M(c\,\adj\,M + A)$ for some $c\in\bZ$ and $A\in\M_n(\bZ)$. If the entries of $A$ are sufficiently small, then multiplication by $cd\,\text{Id} + MA$ has a similar effect on relative vector norms as multiplying by $cd\,\text{Id}$, which has no effect.

We will tailor our choice of $c$ and $A$ so that a coordinate of the simultaneous approximation vector, $(c\,\adj\,M + A)^{-1}\bb$, is $1/\det (c\,\adj M + A)$. This admits Proposition \ref{prop:equivalent} and hence \textsc{gda}.

\subsection{SVP to GDA or SAP}\label{ss:reduction} Algorithm \ref{alg:agrawal} uses the following.

\begin{notation}\renewcommand\arraystretch{1.1}\label{def:matrix}For polynomials $f_1=\sum_{i}f_{1,i}x^i$ and $f_2=\sum_{i}f_{2,i}x^i$ with maximum degree $d$, let $C(f_1,f_2)$ denote the matrix of their coefficients, $$\begin{bmatrix}f_{1,d} &  & 0 & f_{2,d} &  & 0\\ \raisebox{0pt}[0.8\height][0.3\height]{$\vdots$} & \raisebox{0pt}[0.8\height][0.3\height]{$\ddots$} & & \raisebox{0pt}[0.8\height][0.3\height]{$\vdots$} & \raisebox{0pt}[0.8\height][0.3\height]{$\ddots$} & \\ f_{1,1} & \cdots & f_{1,d} & f_{2,1} & \cdots & f_{2,d}\\ f_{1,0} & \cdots & f_{1,d-1} & f_{2,0} & \cdots & f_{2,d-1} \\ & \raisebox{0pt}[0.8\height][0.3\height]{$\ddots$} & \raisebox{0pt}[0.8\height][0.3\height]{$\vdots$} & & \raisebox{0pt}[0.8\height][0.3\height]{$\ddots$} & \raisebox{0pt}[0.8\height][0.3\height]{$\vdots$} \\ 0 & & f_{1,0} & 0 & & f_{2,0}\end{bmatrix}.$$\end{notation}

The matrix above can determine when $f_1$ and $f_2$ are coprime over $\bQ(x)$ in lieu of polynomial long division, where coefficient growth is exponential without complicated mitigations as in \cite{brown}. We demonstrate this now to give some clarity to the meaning behind lines 5 and 6 of Algorithm \ref{alg:agrawal}.

\begin{lemma}\label{lem:matrix}Let $f_1,f_2\in\bZ[x]$, not both constant. As an ideal in $\bZ[x]$, $(f_1,f_2)$ contains $\det C(f_1,f_2)$, which is nonzero if and only if $f_1$ and $f_2$ have no common root in the algebraic closure of $\bQ$.\end{lemma}

\begin{proof}Let $d=\max(\deg f_1,\deg f_2)$. Consider the vector in $\bZ^{2d}$ whose only (perhaps) nonzero entry is $\det C(f_1,f_2)$ in the last coordinate. This is the image under $C(f_1,f_2)$ of some nonzero integer vector. We can split the entries of this vector down the middle to get coefficients for $g_1,g_2\in\bZ[x]$ that have degree at most $d-1$ and satisfy $\det C(f_1,f_2)=f_1g_1+f_2g_2\in (f_1,f_2)$. 

Plugging a common root of $f_1$ and $f_2$ into this last equation, should one exist, shows that $\det C(f_1,f_2)=0$. Conversely, suppose $f_1g_1+f_2g_2=0$ and that $\deg f_1 = d\geq 1$. Then $g_2$ must be nonzero to avoid the same being true of $g_1$, contradicting our choice of nonzero coefficient vector. But $g_2$ has degree at most $d-1$. So $f_1g_1=-f_2g_2$ implies that at least one of $f_1$'s $d$ roots must be shared by $f_2$.\end{proof}

\begin{notation}For a matrix $M$, let $M_{i,j}$ denote the entry in its $i^{\text{th}}$ row and $j^{\text{th}}$ column, and let $\mM_i$ denote its top, left $i\times i$ minor.\end{notation}

Line 1 of the next algorithm requires knowing the position of a nonzero entry in the input matrix, and line 8 requires knowing the maximum magnitude among entries. For notational convenience, we assume that $M_{n,1}$ is the nonzero maximum. 

\begin{algorithm}\caption{A reduction from a short vector problem with $n\geq 2$ to one call to $\textsc{sap}$ (gap-preserving) or $\textsc{gda}$ under a consistent $\ell_p$-norm with $p\in\{1,2,\infty\}$.}\label{alg:agrawal}
\begin{flushleft}
\parbox[t]{\linewidth}{\setlength\parfillskip{0pt}\hspace*{\algorithmicindent}\textbf{input:} $a\hspace{-0.02cm}\geq\hspace{-0.02cm} b\hspace{-0.02cm}\in\hspace{-0.02cm}\bN$ ($\alpha\hspace{-0.02cm}=\hspace{-0.02cm}a/b$), $M\hspace{-0.07cm}\in\hspace{-0.02cm}\M_n(\bZ)$ with $0\hspace{-0.02cm}\neq\hspace{-0.02cm}\det M$ and $M_{n,1}\hspace{-0.08cm}=\hspace{-0.03cm}\max_{i,j}\!|M_{i,j}|$}\\
\hspace*{\algorithmicindent}\textbf{output:} $\bq_0\in\bZ^n$ with $0 < \|M\bq_0\| \leq \alpha\min^{\times}_{\bq\in\bZ^n}\!\|M\bq\|$
\end{flushleft}
\begin{algorithmic}[1]
    \State\Longstate{$p\gets$ least prime not dividing $M_{n,1}\det M$}
    \State $M\gets x\,\adj\,M + p\,\text{Id}$\Longtop{$M=M(x)$ has linear polyno-}
    \For{$i\gets 2$ \textbf{to} $n$}\Bottom{mial entries}
    \State $M_{i,1}\gets M_{i,1}+p$
    \State\Shortstate{$M_{i,i\hspace{-0.02cm}-\hspace{-0.03cm}1}\gets M_{i,i\hspace{-0.02cm}-\hspace{-0.03cm}1}\hspace{-0.07cm}+\hspace{-0.02cm}p^j\!$ with $j\!>\!0$ minimal}\Longtop{need not compute determinant}
    \Statex\hspace{\algorithmicindent}so $\det C((\adj\,\mM_i)_{i,1},(\adj\,\mM_i)_{i,2})\neq 0$\Longbottom{to test each $j$---see Theorem \ref{thm:ag_speed}}
    \Statex \vspace{-\baselineskip}
    \EndFor
    \State $c\gets \det C((\adj\,M)_{n,1},(\adj\,M)_{n,2})$
    \State\Longstate{$c\gets c/p^j$ with $j$ maximal or $p+1$ if $|c|=p^j$}\Top{make $c$ coprime to $p$}
    \State\Longstate{$M\hspace{-0.03cm}\gets\hspace{-0.03cm} M(c^j)$ with $j\hspace{-0.03cm}=\hspace{-0.03cm}\big\lceil\log_{|c|}\!a^2(2M_{n,1}n)^{3n}\big\rceil$}\Longtop{substitute for $x$ so $M\in\text{M}_n(\bZ)$}
    \State\Longstate{$b_1,b_2\gets$ integers with $|b_1|$ minimal so $1=$}\Longtop{that these exist guarantees $M\bx$}
    \Statex $b_1(\adj\,M)_{n,1}+b_2(\adj\,M)_{n,2}$\Bottom{(line 10) and $M$ generate $\bZ^n$}
    \State $\bx\gets M^{-1}(b_1,b_2,0,...,0)$
    \State\Longstate{$q_0\hspace{-0.03cm}\gets\hspace{-0.03cm}\textsc{sap}(\alpha,\bx)$ or $\textsc{gda}(\alpha/n^{1/p},\bx,N)$ with}\Longtop{\textsc{gda} works since $x_n = 1/\det M$,}
    \Statex $N=n^{1/p}\det M/2\alpha$\Bottom{recall Proposition \ref{prop:equivalent}}
    \State \Return $M\{q_0\bx\}$
\end{algorithmic}
\end{algorithm}

Let us turn to the \textbf{for} loop, which builds the matrix Subsection \ref{ss:intuition} called $A$.

\begin{lemma}\label{lem:nonzero}For $i=2,...,n$, there is some $j\leq 2i-2$ satisfying the criterion of line 5 in the \emph{\textbf{for}} loop iteration corresponding to $i$.\end{lemma}

\begin{proof}When $i=2$ we are asked to find $j$ for which the linear polynomials $M_{1,1}$ and $M_{2,1} + p^{j}$ do not share a root (by Lemma \ref{lem:matrix}). The constant term of $M_{1,1}$ is $p$ by line 2, meaning it has at most one root. So asking that $j\leq 2i-2 = 2$ gives enough space to avoid the at-most-one value of $j$ that fails. Now suppose $i\geq 3$ and that the claim holds for $i-1$. Let $M$ be its value after line 4 in the \textbf{for} loop iteration corresponding to $i$, and let $f_1=(\adj\,\mM_{i-1})_{i-1,1}$ and $f_2=(\adj\,\mM_{i-1})_{i-1,2}$.

By assumption there are $g_1,g_2\in\bZ[x]$ with $g_1f_1 + g_2f_2 = \det C(f_1,f_2)\neq 0$. Fix an integer $j$, and let $h_1 = (\adj\,\mM_i)_{i,1} - p^jf_1$ and $h_2 = (\adj\,\mM_i)_{i,2} - p^jf_2$, the polynomials we hope to make coprime with the appropriate choice of $j$. We have $$\begin{bmatrix}f_2 & -f_1\\ g_1 & g_2\end{bmatrix}\begin{bmatrix}h_1\\ h_2\end{bmatrix}=\begin{bmatrix}f_2(\adj\,\mM_i)_{i,1}-f_1(\adj\,\mM_i)_{i,2}\\g_1(\adj\,\mM_i)_{i,1}+g_2(\adj\,\mM_i)_{i,2} - p^j\det C(f_1,f_2)\end{bmatrix}.$$ In the column on the right, where we now focus our attention, $p^j$ has been isolated.

For each root of the top polynomial, there is at most one value of $j$ that makes it a root of the bottom. Thus it suffices to show that $f_2(\adj\,\mM_i)_{i,1}-f_1(\adj\,\mM_i)_{i,2}$ is not the zero polynomial. Then its degree, which is at most $2i-3$, bounds how many values of $j$ can make the right-side polynomials share a root. As this occurs whenever $h_1$ and $h_2$ share a root, Lemma \ref{lem:matrix} would complete the proof.

To show that $f_2(\adj\,\mM_i)_{i,1}-f_1(\adj\,\mM_i)_{i,2}$ is nonzero, we compute its constant term from the following matrix: \begin{equation}\label{eq:4}\renewcommand\arraystretch{0.9}\begin{bmatrix}p & 0 & \cdots & 0 & 0 & \,0 \vspace{0.1cm}\\ p+p^{j_2}\!\! & p & & & 0 & \,0 \vspace{0.1cm}\\ p & p^{j_3}\! & & & & \,0 \vspace{-0.12cm}\\ \vdots & & \ddots & & & \,\vdots \vspace{-0.00cm}\\ p & 0 & & \!\!p^{j_{i-1}}\!\!\! & p & \,0 \vspace{0.1cm}\\ p & 0 & \cdots & 0 & p^{j}\! & \,p\end{bmatrix}.\end{equation} These are the constants in $\mM_i$ after adding $p^j$ in the $i,i-1$ position---the main diagonal comes from line 2, the first column comes from line 4, and the second diagonal comes from line 5. To compute $h_1$ or $h_2$, we use cofactor expansion along the bottom row after deleting the last column and the first or second row. The $(i-2)\times(i-2)$ minor determinants that are multiplied by the bottom row constant $p^j$ are exactly $f_1$ and $f_2$ up to a sign. What remains sums to $(\adj\,\mM_i)_{i,1}$ or $(\adj\,\mM_i)_{i,2}$. So the constant terms of $(\adj\,\mM_i)_{i,1}$, $(\adj\,\mM_i)_{i,2}$, and $f_2$ are $p^{i-1}$, $0$, and $p$ to the power $1+j_3+\cdots j_{i-1}$, respectively. This makes $p$ to the power $i+j_3+\cdots +j_{i-1}$ the constant term of $f_2(\adj\,\mM_i)_{i,1}-f_1(\adj\,\mM_i)_{i,2}$.\end{proof}

We remark that by using a large integer instead of $x$ in line 2, the \textbf{for} loop could successively make pairs of integers coprime rather than polynomials. Then the Euclidean algorithm could test $j$ in line 5; determinants involving polynomial entries need not be computed. We might expect such an algorithm to require $O(n^3\log M_{n,1}n)$ operations (this uses that the average ratio with Euler's phi function, $\varphi(n)/n$, is a positive constant), but the provable worst case is bad. The best current asymptotic upper bound on the size of the interval that must be sieved or otherwise searched to find $j$ is due to Iwaniec \cite{iwaniec}. It only limits the algorithm to $O(n^7\log M_{n,1}n)$ operations. We favored the polynomial approach because of an easier bound on $j$ (Lemma \ref{lem:nonzero}) and a better provable worst case (Theorem \ref{thm:ag_speed}).   

The next lemma allows the vector in line 10 to pass as $\bb$ from Subsection \ref{ss:intuition}.

\begin{lemma}\label{lem:coprime}With $M$ denoting its value in line 9, $\emph{gcd}((\emph{adj}\,M)_{n,1},(\emph{adj}\,M)_{n,2})=1$.\end{lemma} 

\begin{proof}By Lemma \ref{lem:matrix}, it suffices to prove $\text{gcd}((\adj\,M)_{n,1},(\adj\,M)_{n,2},c)=1$ with $c$ as in line 6. Now let $c'$ be $c/p^j$ or $p+1$ as in line 7. Recall the constant terms displayed in (\ref{eq:4}), which show that $(\adj\,M)_{n,2}$ is a power of $p$ modulo $c'$. This implies $\text{gcd}((\adj\,M)_{n,1},(\adj\,M)_{n,2},c)$ is a power of $p$ since $p\nmid c'$. But the constants added throughout the \textbf{for} loop are multiples of $p$. So before substituting for $x$, only the leading coefficient of $(\adj\,M)_{n,1}$ might have been nonzero modulo $p$. With $M$ now the original input matrix, the leading term is $M_{n,1}\!\det M^{n-2}x^{n-1}$. By line 1 this is coprime to $p$ whenever the same is true of the integer substituted for $x$.\end{proof}

\begin{lemma}\label{lem:operator}Let $M$ be the input matrix, let $c^j$ be as in line 8, and let $A$ be such that $c^j\emph{adj}\,M + A$ is Algorithm \ref{alg:agrawal}'s value of $M$ in line 9. Then $\|MA\|_{\emph{op}} < (2nM_{n,1})^{3n}/5n$ under any $\ell_p$-norm.\end{lemma}

\begin{proof}The operator norm is $\max_{\|\bu\|=1}\!\|MA\bu\|$. Using $\|\bu\|_{\infty}\leq 1$ gives \begin{equation}\label{eq:5}\|MA\bu\|\leq n\|MA\bu\|_{\infty}\leq n^2\max_{i,j\in[n]}|(MA)_{i,j}|.\end{equation} 

{\noindent We refer back to (\ref{eq:4}), which displays the entries of $A$ when $i = n$. Lemma \ref{lem:nonzero} says\unskip\parfillskip 0pt \par}

\noindent$j_i\leq 2i-2$, so the entries of $MA$ are bounded in magnitude by \begin{equation}\label{eq:6}\max_{i,j\in[n]}|M_{i,j}|\max(np+p^2,p+p^{2n-2})\leq 2M_{n,1}p^{2n-2}\leq 2M_{n,1}^3p^{2n-2}.\end{equation} (Recall that $n\geq 2$ for this inequality.) Here $np+p^2$ comes from the first column of $A$, and $p+p^{2n-2}$ comes from the $(n-1)^{\text{th}}$ column. 

Now we turn to the size of $p$. If $x\in\bR$ is such that $x\#$, the product of primes not exceeding $x$, is larger than $M_{n,1}|\det M|$, then it must be that $p< x$. Rosser and Schoenfeld's lower bound on Chebyshev's theta function, $\vartheta(x)=\sum_{p\leq x}\!\log p$, gives $\vartheta(x) > 0.231x$ when $x\geq 2$ \cite{rosser}. For the determinant we use Hadamard's bound: $|\det M|\leq (M_{n,1}\sqrt{n})^n$ \cite{hadamard}. So take $x=(\log M_{n,1}^{3n}n^n)/0.462$ (note that $x\geq 2$ even when $n=2$ and $M_{n,1}=1$, allowing for the Rosser-Schoenfeld bound) to get $$\log x\# =\vartheta(x) > 0.231x=\tfrac{1}{2}\log M_{n,1}^{3n}n^{n}\geq \log M_{n,1}^{n+1}n^{n/2} \geq \log M_{n,1}|\det M|.$$ 

Combining $p<x$ with (\ref{eq:5}) and (\ref{eq:6}) gives $\|MA\|_{\text{op}} < 2M_{n,1}^3n^2x^{2n-2}$. We must show that this is less than the stated bound of $(2nM_{n,1})^{3n}/5n$. To do this, raise both expressions to the power $1/(n-1)$ and use $(5/4)^{1/(n-1)}\leq 5/4$. This simplifies the desired inequality to $(\log M_{n,1}^3n)^2<1.366M_{n,1}^3n$, which is true.\end{proof}

\begin{theorem}\label{thm:ag_works}Under the $\ell_1$, $\ell_2$, or $\ell_{\infty}$-norm, the output of Algorithm \ref{alg:agrawal} solves the initial short vector problem.\end{theorem}

\begin{proof} There are two parts to the proof: 1) showing that the algorithm replaces the columns of $M$ with $n+1$ vectors that define the same lattice, $n$ of them being nearly scaled orthonormal, and 2) showing that nearly scaled orthonormal is as good as being scaled orthonormal. Throughout the proof, let $M$ be the input matrix, let $c^j$ be as in line 8, let $M'$ be Algorithm \ref{alg:agrawal}'s value of $M$ in line 9, and let $A=M'-c^j\hspace{0.03cm}\adj\,M$ be the matrix of constants added throughout the \textbf{for} loop (as used in Lemma \ref{lem:operator} and as shown in (\ref{eq:4}) when $i=n$). 

For part 1), with $\bb = (b_1,b_2,0,...,0)$ from line 10, Lemma \ref{lem:coprime} gives \begin{equation}\label{eq:12}\bx=M'^{-1}\bb=\frac{(x_1,x_2,...,1)}{\det M'}.\end{equation} By Cramer's rule \cite{cramer}, the $1$ in the last coordinate is the determinant after replacing the last column of $M'$ by $\bb$, so that these $n$ columns generate $\bZ^n$. This in turn shows that the columns of $MM'$ and $M\bb$ generate the input lattice. Also note by Proposition \ref{prop:equivalent}, that a coordinate of $\det M'\,\bx$ being $1$ allows for \textsc{gda} in place of \textsc{sap} with $N$ set to $n^{1/p}\det M'/2\alpha$ and $\alpha$ scaled by $1/n^{1/p}$. 

Instead of finding a short integer combination of $M\bb$ and the columns of \begin{equation}\label{eq:7}MM'=c^j\!\det M\,\text{Id} + MA,\end{equation} Algorithm \ref{alg:agrawal} uses $(MM')^{-1}(M\bb)= \bx$ and the columns of $(MM')^{-1}(MM')=\text{Id}$. Then $MM'\{q_0\bx\}$ is proposed as a short vector. It is indeed an element of the original lattice as the coordinates of $M'\{q_0\bx\}\equiv q_0\bb\text{ mod }\bZ^n$ are all integers. But it must be checked is that $MM'\{q_0\bx\}$ is short whenever $\{q_0\bx\}$ is. Part 2) of the proof is to make precise the insignificance of the second matrix summand, $MA$, in (\ref{eq:7}). We begin by computing how much multiplication by the full matrix in (\ref{eq:7}) is allowed to inflate the gap without invalidating the output of \textsc{gda} or \textsc{sap}.

By Minkowski's theorem \cite{minkowski}, the magnitude of the shortest vector in the original lattice with respect to the $\ell_{\infty}$-norm is not more than $|\!\det M|^{1/n}$. So under an $\ell_p$-norm with $p\in\bN$, the shortest vector has some magnitude, say $\lambda$, with $(n^{1/p}|\det M|^{1/n})^p\geq \lambda^p\in\bZ$. In particular, $n|\!\det M|^{2/n}\geq \lambda^2\in\bZ$ when $p\in\{1,2,\infty\}$. Now, if $\bq\in \bZ^n$ is such that $\|M\bq\|^2 < (a^2\lambda^2+1)/b^2$, then it must be that $\|M\bq\| \leq a\lambda/b$ since there are no integers strictly between $(a\lambda/b)^2$ and $(a^2\lambda^2+1)/b^2$. Thus multiplication by $MM'$ may harmlessly inflate the gap between the norms of our output vector and shortest vector by anything less than \begin{equation}\label{eq:8}\frac{\sqrt{a^2\lambda^2+1}}{b\alpha\lambda} = \frac{\sqrt{a^2\lambda^2+1}}{a\lambda} \geq \frac{\sqrt{a^2n|\det M|^{2/n}+1}}{a\sqrt{n}|\det M|^{1/n}}.\end{equation}

Scaling does not affect the ratio of vector norms, so to determine the effect of multiplication by (\ref{eq:7}) it suffices to consider the matrix \begin{equation}\label{eq:9}\text{Id} + MA/c^j\!\det M\end{equation} instead. If $\bq_{\min}$ is a shortest nonzero vector in the simultaneous approximation lattice generated by $\bZ^n$ and $\bx$, then a shortest vector after applying (\ref{eq:9}) to this lattice has norm at least $(1 - \|MA\|_{\text{op}}/|c^j\!\det M|)\|\bq_{\min}\|$. Similarly, the vector $\{q_0\bx\}$ obtained using $q_0$ from line 11 increases in norm by at most a factor of $(1 + \|MA\|_{\text{op}}/|c^j\!\det M|)$. Combining this with our conclusion regarding (\ref{eq:8}) shows that it suffices to verify the following inequality holds: \begin{equation}\label{eq:10}\frac{1+\|MA\|_{\text{op}}/|c^j\!\det M|}{1 - \|MA\|_{\text{op}}/|c^j\!\det M|}\leq\frac{\sqrt{a^2n|\!\det M|^{2/n}+1}}{a\sqrt{n}|\!\det M|^{1/n}}.\end{equation}

Now solve for $|c^j|$ to get a lower bound of $$\frac{\sqrt{a^2n|\!\det M|^{2/n}+1}+a\sqrt{n}|\!\det M|^{1/n}}{\sqrt{a^2n|\!\det M|^{2/n}+1}-a\sqrt{n}|\!\det M|^{1/n}}\cdot \frac{\|MA\|_{\text{op}}}{|\!\det M|}< \frac{(5a^2n|\!\det M|^{2/n})\|MA\|_{\text{op}}}{|\!\det M|}.$$

Ignoring the powers of $|\det M|$ on the right-hand side since $2/n\leq 1$, we see that $j$ in line 8 is chosen to make the bound above agree exactly with Lemma \ref{lem:operator}.\end{proof}

\begin{theorem}\label{thm:ag_speed}Let $m=\max(a^{1/n^3}\!,M_{n,1})$. The reduction in Algorithm \ref{alg:agrawal} requires $O(n^4\log mn)$ operations on integers of length $O(n^4\log mn)$.\end{theorem}

\begin{proof}We will use that finding determinants, adjugates, inverses, or characteristic polynomials of $n\times n$ matrices with entry magnitudes bounded by $m$ requires $O(n^3)$ operations on integers of length $O(n\log mn)$. For example, see Danilevskii's method for the characteristic polynomial \cite{danilevskii} and the Bareiss algorithm for the others \cite{bareiss}. Note that we may then compute determinants of matrices with linear polynomial entries in $O(n^3)$ operations provided the matrix of linear terms or the matrix of constant terms is invertible.

In the proof of Lemma \ref{lem:operator} we showed that the prime $p$ from line 1 is less than $(\log M_{n,1}^{3n}n^n)/0.462$. So finding it does not contribute to asymptotic complexity.

Now consider the \textbf{for} loop, where we must avoid recomputing the determinant in line 5 for each value of $j$ in order to meet the prescribed bound on operations.

Let $i\geq 3$ and fix some notation: $M$ is its value after line 4, $f_1\!=\!(\adj\,\mM_{i\hspace{-0.02cm}-\hspace{-0.02cm}1})_{i\hspace{-0.02cm}-\hspace{-0.02cm}1,1}$ and $f_2\!=\!(\adj\,\mM_{i\hspace{-0.02cm}-\hspace{-0.02cm}1})_{i\hspace{-0.02cm}-\hspace{-0.02cm}1,2}$, $g_1 $ and $g_2$ have degree at most $i\hspace{-0.05cm}-\hspace{-0.05cm}3$ and $f_1g_1\!+\!f_2g_2=$ $\det C(f_1,f_2)\!\neq\! 0$, and for some $j$, $h_1\!=\!(\adj\,\mM_{i})_{i,1}\! -\! p^jf_1$ and $h_2\!=\!(\adj\,\mM_{i})_{i,2}\! -\! p^jf_2$. Note for computing $(\adj\,\mM_{i})_{i,2}$ that the constant term matrix is not invertible (see (\ref{eq:4})), which may also be true of the linear term matrix. Because this complicates combining the Bareiss and Danivelskii algorithms, we could find $(\adj\,\mM_{i})_{i,2}$ indirectly by computing $h_2$ for two values of $j$ that produce an invertible constant term matrix (recall from (\ref{eq:4}) that $f_2$ has nonzero constant term), then solve for it.

Call the polynomials in the resulting column vector below $h'_1$ and $h'_2$: \begin{equation}\label{eq:11}\begin{bmatrix}f_2+p^jg_1x^{2i-3} & -f_1+p^jg_2x^{2i-3}\\ g_1 & g_2\end{bmatrix}\begin{bmatrix}h_1\\ h_2\end{bmatrix}\end{equation}\vspace{-0.1cm} $$=\begin{bmatrix}f_2(\adj\,\mM_{i})_{i,1}-f_1(\adj\,\mM_{i})_{i,2} - p^{j}\det C(f_1,f_2)x^{2i-3}\\g_1(\adj\,\mM_{i})_{i,1} + g_2(\adj\,\mM_{i})_{i,2} - p^{j}\det C(f_1,f_2)\end{bmatrix}.$$\vspace{-0.1cm}

\noindent Remark that if $j$ makes $h'_1$ and $h'_2$ avoid a common root, it does so for $h_1$ and $h_2$.

{View $C(h'_1,h'_2)$ as a matrix with linear polynomial entries where $p^j$ is the variable. This variable only appears in the leading term of $h'_1$ and the constant term of $h'_2$. So $p^j$ only occurs on the main diagonal of $C(h'_1,h'_2)$, where its coefficient is nonzero. In particular, the polynomial $\det C(h'_1,h'_2)$ can be found in $O(n^3)$ operations. Substituting different values of $p^j$ into this polynomial until one is nonzero avoids repeatedly finding determinants. And note that we still need only test up to $j=2i-2$ as stated in Lemma \ref{lem:nonzero} because the determinant of the matrix in (\ref{eq:11}) is a constant (a unit in $\bQ(x)$). Thus each \textbf{for} loop iteration requires $O(n^3)$ operations.\unskip\parfillskip 0pt \par}

The integers composing the linear polynomial matrix entries that begin each \textbf{for} loop iteration are small powers of $p=O(n\log M_{n,1}n)$ and entries in the adjugate of the input matrix, $M$. By Hadamard's bound they are thus $O(n\log M_{n,1}n)$ in length. Hadamard's bound also applies to the coefficients of $(\adj\,\mM_{i})_{i,1}$ and $(\adj\,\mM_{i})_{i,2}$, making their lengths $O(n^2\log M_{n,1}n)$. And it applies again to make $\det C((\adj\,\mM_{i})_{i,1}, (\adj\,\mM_{i})_{i,2}))$ have length $O(n^3\log M_{n,1}n)$. This is our bound on the length of $c$ in line 6 and hence the length of $c$ in line 7. The length of $c^j$ in line 8 is then $O(\max(\log a^2(2M_{n,1}n)^{3n}, \log |c|))=O(n^3\log mn)$, with the maximum accommodating the ceiling function. Then a final application of Hadamard's bound for lines 9 and 10 makes integer lengths $O(n^4\log mn)$. This is therefore a bound on the number of operations required by the Euclidean algorithm in line 9.\end{proof}

{In \cite{dinur}, Dinur proves the NP-hardness of short vector problems under the $\ell_{\infty}$-norm when $\alpha=n^{c/\log \log n}$ for some $c>0$ by giving a direct reduction from the Boolean satisfiability problem (\textsc{sat}). As a consequence, Theorems \ref{thm:ag_works} and \ref{thm:ag_speed} prove the same for both good Diophantine approximation and simultaneous approximation problems. (There is no gap inflation for $\textsc{gda}$ in line 11 under the $\ell_{\infty}$-norm.)\unskip\parfillskip 0pt \par}

\begin{corollary}\label{cor:nphard}Good Diophantine approximation and simultaneous approximation problems are \emph{NP}-hard under the $\ell_{\infty}$-norm with $\alpha=n^{c/\log \log n}$ for some $c>0$.\qed\end{corollary}

This result is known for good Diophantine approximation \cite{chen}, though the reduction $\textsc{sat}\to\textsc{svp}\to\textsc{gda}$ completed here is simpler. Chen and Meng adapt the work of Dinur as well as R\"{o}ssner and Seifert \cite{rossner2} to reduce \textsc{sat} to finding short integer vectors that solve a homogeneous system of linear equations (\textsc{hls}) via an algorithm from \cite{arora}, which changes the problem to finding pseudo-labels for a regular bipartite graph (\textsc{psl}). The number of equations in the \textsc{hls} system is then decreased to one (now called \textsc{sir}), wherefrom a reduction to \textsc{gda} has been known \cite{rossner}. Each link, $\textsc{sat}\to\textsc{psl}\to\textsc{hls}\to\textsc{sir}\to\textsc{gda}$, is gap-preserving under the $\ell_{\infty}$-norm.

Short vector problems are only known to be NP-hard under the $\ell_{\infty}$-norm. But there are other hardness results under a general $\ell_p$-norm for which Theorems \ref{thm:ag_works} and \ref{thm:ag_speed} can be considered complementary. See \cite{kumar} for an exposition.

Another corollary is the reduction from a simultaneous approximation problem to \textsc{gda}, giving the final row of Table \ref{tab:results}. By Proposition \ref{prop:lag_speed}, Algorithm \ref{alg:lagarias} results in one call to \textsc{svp} with integers of length $O(n\log m)$, where we can take $m$ to be the maximum magnitude among $a^{1/n^4}$ (still $\alpha=a/b$) and the integers defining $\bx$. Then Theorem \ref{thm:ag_speed} implies the reduction to \textsc{gda} requires $O(n^4\log m^nn)=O(n^5\log m)$ (absorbing the operations required by Algorithm \ref{alg:lagarias}) on integers of length $O(n^5\log m)$.

\subsection{Further discussion}\label{ss:discuss} The last algorithm was restricted to an $\ell_p$-norm for $p\in\{1,2,\infty\}$. So we will discuss what happens with a more general approach.

Multiplication by $MM'$, shown in (\ref{eq:9}), may change the gap between the length of the shortest vector in the simultaneous approximation lattice and that of the vector output by $\textsc{gda}$ or $\textsc{sap}$. That this potential inflation does not invalidate our output relies on the set of vector norms being discrete and $\alpha$ being rational---facts that were exploited to produce the expression in (\ref{eq:8}). The idea behind that paragraph is to find a nonempty interval $(\alpha\lambda, \alpha'\lambda)$, where $\lambda=\min^{\times}_{\bq\in\bZ^n}\!\|M\bq\|$, that contains no norms from the lattice defined by $M$ (or even $\bZ^n$ for the interval tacitly given in the proof). This creates admissible inflation, $\alpha'/\alpha$, which is (\ref{eq:8}).

The purpose of restricting to $\ell_1$, $\ell_2$, or $\ell_{\infty}$ is to facilitate finding this interval. Knowing that $(b\alpha\lambda)^2\in\bZ$ for some $b\in\bZ$ simplifies the search for $\alpha'$. The same is true for any $\ell_p$-norm with $p\in\bN$. But the immediate analogs of (\ref{eq:8}), (\ref{eq:9}), and (\ref{eq:10}) lead to a replacement for the very last bound used in the proof of the form $$\frac{(5pa^pn|\det M|^{p/n})\|MA\|_{\text{op}}}{2|\det M|}.$$ This makes the number of operations needed to execute line 9 depend exponentially on the input length $\log p$ (though it is still polynomial for any \textit{fixed} $p$). We have not taken into account, however, the possibility of a nontrivial lower bound for the difference between large consecutive integers which are sums of $n$ perfect $p^\text{th}$ powers. Such a bound would allow for a longer interval, $(\alpha\lambda, \alpha'\lambda)$, that provably contains no lattice norms.

These arguments are all in effort to perfectly preserve the gap when reducing to \textsc{sap} or, when $p=\infty$, \textsc{gda}. The situation clarifies if a small amount of inflation is allowed. To solve a short vector problem with gap $\alpha$ using $\textsc{sap}$ with gap $\alpha' < \alpha$, inequality (\ref{eq:10}) becomes $$\frac{1+\|MA\|_{\text{op}}/|c^j\det M|}{1 - \|MA\|_{\text{op}}/|c^j\det M|}\leq\frac{\alpha}{\alpha'}.$$ We still need to substitute a power of $c$ for $x$ in line 8 for the purpose of Lemma \ref{lem:coprime}. Given these two constraints, it is sufficient to take $M\leftarrow M(c^j)$ for $$j=\left\lceil\log_{|c|}\frac{(\alpha + \alpha')\|MA\|_{\text{op}}}{(\alpha - \alpha')|\!\det M|}\right\rceil,$$ which can be made more explicit with Lemma \ref{lem:operator}. There is no need to insist that $\alpha$ is rational or impose a restriction on $p\in[1,\infty]$ defining the norm.

As a final note, the reduction to \textsc{sap} again adapts to inhomogeneous forms of these problems while the reduction to \textsc{gda} does not. If $\by\in\bQ^n$, then the algorithm (which now reduces the \textit{closest} vector problem) can end by solving the simultaneous approximation problem of finding $q_0\in\bZ$ with $\|\{q_0\bx-(MM')^{-1}\by\}\|\leq\alpha\min_{q\in\bZ}^{\times}\!\|\{q\bx-(MM')^{-1}\by\}\|$, using the matrix from (\ref{eq:9}). But unless we know that the last coordinate (where the $1$ is located in (\ref{eq:12})) of $(MM')^{-1}\by$ is an integer, there is no clear modification to Proposition \ref{prop:equivalent} that permits the use of \textsc{gda}.

\bibliographystyle{plain}
\bibliography{refs}

\begin{thebibliography}{10}

\bibitem{agrawal}
Manindra Agrawal.
\newblock {S}imultaneous {D}iophantine approximation and short lattice vectors.
\newblock \url{https://www.youtube.com/watch?v=7SGCXbim6Ug}, 2019.
\newblock Accessed: 2019-12-01.

\bibitem{armknecht}
Frederik Armknecht, Carsten Elsner, and Martin Schmidt.
\newblock Using the inhomogeneous simultaneous approximation problem for
  cryptographic design.
\newblock In {\em International Conference on Cryptology in Africa}, pages
  242--259. Springer, 2011.

\bibitem{arora}
Sanjeev Arora, L{\'a}szl{\'o} Babai, Jacques Stern, and Elizabeth Sweedyk.
\newblock The hardness of approximate optima in lattices, codes, and systems of
  linear equations.
\newblock {\em Journal of Computer and System Sciences}, 54(2):317--331, 1997.

\bibitem{baocang}
Wang Baocang and Hu~Yupu.
\newblock Public key cryptosystem based on two cryptographic assumptions.
\newblock {\em IEE Proceedings-Communications}, 152(6):861--865, 2005.

\bibitem{bareiss}
Erwin~H. Bareiss.
\newblock Sylvester’s identity and multistep integer-preserving {G}aussian
  elimination.
\newblock {\em Mathematics of Computation}, 22(103):565--578, 1968.

\bibitem{bernstein}
Daniel~J. Bernstein and Tanja Lange.
\newblock Post-quantum cryptography.
\newblock {\em Nature}, 549(7671):188--194, 2017.

\bibitem{brentjes}
Arne~Johan Brentjes.
\newblock Multi-dimensional continued fraction algorithms.
\newblock {\em MC Tracts}, 1981.

\bibitem{brown}
W.~Steven Brown.
\newblock On {E}uclid's algorithm and the computation of polynomial greatest
  common divisors.
\newblock {\em Journal of the ACM}, 18(4):478--504, 1971.

\bibitem{chen}
Wenbin Chen and Jiangtao Meng.
\newblock An improved lower bound for approximating {S}hortest {I}nteger
  {R}elation in $\ell_{\infty}$-norm ({$\text{SIR}_{\infty}$}).
\newblock {\em Information Processing Letters}, 101(4):174--179, 2007.

\bibitem{cramer}
Gabriel Cramer.
\newblock {\em Introduction \`{a} l'analyse des lignes courbes
  alg\'{e}briques}.
\newblock Chez les Fr\`{e}res Cramer \& Cl. Philibert, 1750.

\bibitem{danilevskii}
A.~M. Danilevskii.
\newblock On the numerical solution of the secular equation.
\newblock {\em Matematicheskii Sbornik}, 44(2):169--172, 1937.

\bibitem{dinur}
Irit Dinur.
\newblock Approximating {$\text{SVP}_{\infty}$} to within almost-polynomial
  factors is {NP}-hard.
\newblock {\em Theoretical Computer Science}, 285(1):55--71, 2002.

\bibitem{frank}
Andr{\'a}s Frank and {\'E}va Tardos.
\newblock An application of simultaneous {D}iophantine approximation in
  combinatorial optimization.
\newblock {\em Combinatorica}, 7(1):49--65, 1987.

\bibitem{galbraith}
Steven~D. Galbraith.
\newblock {\em Mathematics of public key cryptography}.
\newblock Cambridge University Press, 2012.

\bibitem{hadamard}
Jacques Hadamard.
\newblock R\'{e}solution d'une question relative aux determinants.
\newblock {\em Bulletin des Sciences Math\'{e}matiques}, 2:240--246, 1893.

\bibitem{inoue}
Hiroshi Inoue, Shoichi Kamada, and Koichiro Naito.
\newblock Simultaneous approximation problems of $p$-adic numbers and $p$-adic
  knapsack cryptosystems---{A}lice in $p$-adic numberland.
\newblock {\em $p$-Adic Numbers, Ultrametric Analysis, and Applications},
  8(4):312--324, 2016.

\bibitem{iwaniec}
Henryk Iwaniec.
\newblock On the problem of {J}acobsthal.
\newblock {\em Demonstratio Mathematica}, 11(1):225--232, 1978.

\bibitem{kannan}
Ravindran Kannan and Achim Bachem.
\newblock Polynomial algorithms for computing the {S}mith and {H}ermite normal
  forms of an integer matrix.
\newblock {\em Siam Journal on Computing}, 8(4):499--507, 1979.

\bibitem{kumar}
Ravi Kumar and D.~Sivakumar.
\newblock Complexity of {SVP}--a reader’s digest.
\newblock {\em SIGACT News}, 32(3):40--52, 2001.

\bibitem{lagarias2}
Jeffrey~C. Lagarias.
\newblock Knapsack public key cryptosystems and {D}iophantine approximation.
\newblock In {\em Advances in cryptology}, pages 3--23. Springer, 1984.

\bibitem{lagarias}
Jeffrey~C. Lagarias.
\newblock The computational complexity of simultaneous {D}iophantine
  approximation problems.
\newblock {\em SIAM Journal on Computing}, 14(1):196--209, 1985.

\bibitem{lenstra}
Hendrik~W. Lenstra, Arjen~K. Lenstra, and L\'{a}szl\'{o} Lov\'{a}sz.
\newblock Factoring polynomials with rational coefficients.
\newblock {\em Mathematische Annalen}, 264(4):515--534, 1982.

\bibitem{goldwasser}
Daniele Micciancio and Shafi Goldwasser.
\newblock {\em Complexity of lattice problems: a cryptographic perspective},
  volume 671.
\newblock Springer Science \& Business Media, 2012.

\bibitem{minkowski}
Hermann Minkowski.
\newblock {\em Geometrie der zahlen}, volume~40.
\newblock Leipzig and Berlin: R. G. Teubner, 1910.

\bibitem{nguyen}
Phong~Q. Nguyen.
\newblock Lattice reduction algorithms: {T}heory and practice.
\newblock In {\em Annual International Conference on the Theory and
  Applications of Cryptographic Techniques}, pages 2--6. Springer, 2011.

\bibitem{peikert}
Chris Peikert.
\newblock A decade of lattice cryptography.
\newblock {\em Foundations and Trends in Theoretical Computer Science},
  10(4):283--424, 2016.

\bibitem{pohst}
Michael~E. Pohst.
\newblock A modification of the {LLL} reduction algorithm.
\newblock {\em Journal of Symbolic Computation}, 4(1):123--127, 1987.

\bibitem{rosser}
J.~Barkley Rosser and Lowell Schoenfeld.
\newblock Approximate formulas for some functions of prime numbers.
\newblock {\em Illinois Journal of Mathematics}, 6(1):64--94, 1962.

\bibitem{rossner}
Carsten R{\"o}ssner and Jean-Pierre Seifert.
\newblock Approximating good simultaneous {D}iophantine approximations is
  almost {NP}-hard.
\newblock In {\em International Symposium on Mathematical Foundations of
  Computer Science}, pages 494--505. Springer, 1996.

\bibitem{rossner2}
Carsten R{\"o}ssner and Jean-Pierre Seifert.
\newblock On the hardness of approximating shortest integer relations among
  rational numbers.
\newblock {\em Theoretical Computer Science}, 209(1-2):287--297, 1998.

\bibitem{shamir}
Adi Shamir.
\newblock A polynomial time algorithm for breaking the basic {M}erkle-{H}ellman
  cryptosystem.
\newblock In {\em 23rd Annual Symposium on Foundations of Computer Science
  (sfcs 1982)}, pages 145--152. IEEE, 1982.

\end{thebibliography}

\end{document}